\documentclass[12pt]{amsart}

\usepackage{amssymb}
\usepackage{latexsym}
\usepackage{enumerate}

\numberwithin{equation}{section}

\makeatletter
\@namedef{subjclassname@2010}{%
  \textup{2010} Mathematics Subject Classification}
\makeatother

\def\neg1{\text{\boldmath$1$}}

\def\cC{\mathcal C}

\def\cH{\mathcal H}
\def\cK{\mathcal K}
\def\cL{\mathcal L}

\def\cX{\mathcal X}

\def\bT{\mathbf T}

\def\Fp{{\mathbb{F}_p}}
\def\Fq{{\mathbb{F}_q}}
\def\Fql{\mathbb{F}_{q^\ell}}

\newtheorem{theorem}{Theorem}[section]

\newtheorem{lemma}[theorem]{Lemma}

{\theoremstyle{definition}

\newtheorem{question}[theorem]{Question}
}

{\theoremstyle{remark}
\newtheorem{remark}[theorem]{Remark}
}

\begin{document}
		
		\baselineskip=17pt


    \title[Complete arcs and curves]{Complete arcs arising from a 
generalization \\of the Hermitian curve}

\author[H.\, Borges]{Herivelto Borges} 
  \address{Universidade de S\~ao Paulo, 
Instituto de Ci\^encias Matem\'aticas e de Computa\c{c}\~ao, 
S\~ao Carlos, SP 13560-970, Brazil}
   \email{hborges@icmc.usp.br}

\author[B.\, Motta]{Beatriz Motta}
  \address{Departamento de Matem\'atica, Instituto de Ci\^encias Exatas, 
Universidade Federal de Juiz de Fora, Rua Jos\'e Louren\c{c}o Kelmer, 
s/n - Campus Universit\'ario, Bairro S\~ao Pedro, 36036-900, Juiz de 
Fora, MG, Brazil} 
  \email{beatriz@ice.ufjf.br}

\author[F.\, Torres]{Fernando Torres}
  \address{University of Campinas (UNICAMP), Institute of Mathematics, 
Statistics and Computer Science (IMECC), R. 
S\'ergio Buarque de Holanda, 651, Cidade 
Universit\'aria \lq\lq Zeferino Vaz", 13083-059, Campinas, SP, Brazil}
  \email{ftorres@ime.unicamp.br}


   \begin{abstract} We investigate complete 
arcs of degree greater than two, in projective 
planes over finite fields, arising from the set 
of rational points of a generalization of the 
Hermitian curve. The degree of the arcs is 
closely related to the number of rational 
points of a class of Artin--Schreier curves 
which is calculated by using exponential sums 
via Coulter's approach. We also single out some 
examples of maximal curves.
   \end{abstract}
	

\maketitle

    \section{Introduction}\label{s1}

Let $\Fq$ and $PG(2,q)$ denote the finite field of order $q$ and the
projective plane over $\Fq$, respectively. A pointset $\cK\subseteq PG(2,q)$ of size $N$ is
called an {\em arc} of {\em degree} $d$ or, simply, an {\em $(N,d)$-arc} if no
line of $PG(2,q)$ meets $\cK$ in more than $d$ points. The $(N,d)$-arc $\cK$
is called {\em complete} if it is not contained in an $(N+1,d)$-arc; that is,
if for every point $P\in PG(2,q)\setminus\cK$ there is a line through $P$
meeting $\cK$ in exactly $d$ points. A basic problem in Finite Geometry is
the existence and uniqueness of complete arcs. For basic facts on these
objects, the reader is referred to the book {\cite{hirschfeld}} by Hirschfeld.

Throughout this paper by a {\em plane curve} we 
shall mean a projective, geometrically 
irreducible plane curve. Let $\cX$ be a plane 
curve of degree $d$ defined over $\Fq$. The set of the 
$\Fq$-rational points of 
$\cX$ in $PG(2,q)$, denoted by $\cX(\Fq)$, is a natural example of an 
$(N,d)$-arc with $N=\#\cX(\Fq)$ (B{\'e}zout's 
Theorem). The problem of the completeness of 
$\cX(\Fq)$ as an $(N,d)$-arc was raised by 
Hirschfeld and Voloch in 1988 
{\cite{hirschfeld-voloch}}. For instance, if 
$\cX$ is a conic in odd characteristic or the 
Hermitian curve, namely the plane curve defined 
by the affine equation $y^{q+1}=x^q+x$ over 
$\mathbb F_{q^2}$, then the set of rational 
points of such a curve is an example of a 
complete arc; see for example {\cite[Lemma 7.20, Ch. 
8]{hirschfeld}}. A generalization of the 
Hermitian curve is given by an 
$\Fq$-{\em Frobenius nonclassical} plane curve; 
that is, a plane curve over $\Fq$ such that the 
$\Fq$-Frobenius map takes each nonsingular 
point of the curve to the tangent line at that 
point (cf. {\cite{hefez-voloch}}). Such curves 
are usually equipped with a large number of 
rational points (loc. cit.) so that one can 
expect to handled examples of complete arcs of 
large size compared with their degrees. 
Recently Giulietti et al. {\cite{italianos}} and 
Borges {\cite{hb}} studied the set of 
$\Fq$-rational points of further examples of 
$\Fq$-Frobenius nonclassical plane curves that 
also give rise to complete arcs. For background 
on curves over finite fields we refer to the 
book {\cite{HKT}}.

Any $(N,d)$-arc arising 
from a plane curve gives rise to an algebraic geometry (AG) code 
with parameters $[N,3,N-d']$, $d'\leq d$; see, for example, 
 {\cite[Sect. 3.1.1]{TV}}. Here, if the arc is 
complete, the corresponding code has minimum 
distance equal to $N-d$ and it cannot be 
extended to a code with larger minimum 
distance. This is analogous to the well known 
relation between complete $(N,2)$-arcs and 
non-extendable MDS codes (loc. cit.).

In this article we investigate $(N,d)$-arcs 
derived from the set of rational points of a 
Frobenius nonclassical curve introduced by 
Borges and Concei\c c\~ao in {\cite{hb2}} (see 
Section \ref{s2} here) and which is a natural 
generalization of a Hermitian curve. Our main 
result is Theorem \ref{main-result}. The 
computation of the degree of the corresponding 
arcs is closely related to
 the study of rational points of a class of Artin-Schreier curves, see 
(\ref{n=qm}); here Coulter's approach 
{\cite{c1}},{\cite{c2}},{\cite{c4}},{\cite{c3}} is used. By taking advantage of 
the aforementioned computation regarding rational points, we slightly 
extend some results of Wolfmann {\cite{wolfmann}} and Coulter {\cite{c3}} by 
pointing out some examples of maximal curves of Artin--Schreier type; see 
Theorem \ref{max}.

    \section{The curve $\cH$}\label{s2}

Let $q$ be a power of a prime $p$. Let $\ell$ be 
an integer with $\ell\geq 2$ and 
define $r=r(\ell)$ as  
the smallest integer $r\geq\ell/2$ such that $\gcd(\ell,r)=1$; that is,
   \begin{equation}\label{r1}
r= \begin{cases} 
1, & \text{ if }\quad \ell=2\\
\ell/2+1, & \text{ if }\quad \ell\equiv0\pmod4\\
\ell/2+2, & \text{ if }\quad \ell\geq 6,\, \ell\equiv2\pmod4\\
(\ell+1)/2, & \text{ if }\quad \ell \text{ is odd}\, .
   \end{cases} 
  \end{equation}
For a symbol $z$, set 
   $$
\bT(z):=z^{q^{\ell-1}}+z^{q^{\ell-2}}+\cdots+z\, .
   $$ 
In particular, $\bT: \Fql\to\Fq$ denotes the trace map from 
$\Fql$ to $\Fq$. In {\cite{hb2}} the plane curve $\cH$ defined 
by the affine equation
   $$
\bT(y)=\bT(x^{q^r+1})\,\quad \pmod{x^{q^{\ell}}-x}
   $$
over $\Fql$ was considered. The main properties of this curve are 
listed below.
   \begin{theorem}{\rm ({\cite{hb2}})}\label{herivelto} Let $\ell$ and $r$ be as abose. Suppose $p>2$ if $\ell=2$. The curve $\cH$ has 
degree $q^{\ell-1}+q^{r-1}$, genus $q^r(q^{\ell-1}-1)/2$ 
and its number of $\Fql$-rational points in 
$PG(2,q^{\ell})$ is $q^{2\ell-1}+1$. It has 
just one point at infinity of projective coordinates $(X:Y:Z)=(0:1:0)$ 
which is also its only singular point whenever $\ell\geq 3$. Furthermore, the 
curve is $\Fql$-Frobenius nonclassical. 
   \end{theorem}
Notice that the number of $\Fql$-rational points of the nonsingular 
model of $\cH$ is also $q^{2\ell-1}+1$ (loc. cit.). If $p>2$ and $\ell=2$, 
then it is clear that $\cH$ is the Hermitian curve and 
thus $\cH(\mathbb F_{q^2})$ is a well known complete $(q^3+1,q+1)$-arc. 
Here we focus on the more complicated case $\ell\geq 3$. 
   \begin{remark} In {\cite{gshermitiana}}, Garcia and Stichtenoth 
considered the plane curve $\cC$ defined by the affine equation 
   $$
y^{q^{\ell-1}}+\dots+y^{q}+y=x^{q^{\ell-1}+q^{\ell-2}}+\dots+x^{q+1}
  $$ 
over $\Fql$ with $\ell\geq 2$; see also {\cite{bul,castle,mst}}. This 
curve has degree $q^{\ell-1}+q^{\ell-2}$, 
genus $q^{\ell-1}(q^{\ell-1}-1)/2$ and $q^{2\ell-1}+1$ $\Fql$-rational 
points in $PG(2,\Fql)$. The nonsingular model of $\cC$ also has  
$q^{2\ell-1}+1$ rational points over $\Fql$ (loc. cit.).

For $\ell=2$ and $p>2$ both plane curves $\cC$ and $\cH$ are the 
Hermitian curve. For $\ell=3$, they define the same curve. For $\ell=4$ 
and $\ell=6$, their degrees, genus and numbers of rational points are 
the same. In general, the 
number of its rational points 
coincide; however, the degree and genus of $\cH$ are smaller than that 
of $\cC$. In particular, the ratios (number of rational points)/degree  
and (number of rational points)/genus are 
better on the curve $\cH$. Such rates are particularly 
important; for example, in the context of Finite Geometry or Coding Theory 
via AG codes; see, for example, {\cite{HKT}}.
   \end{remark}
   
As mentioned in the Introduction, the main goal of this paper is 
the study of the arc $\cK:=\cH(\Fql)$ in 
$PG(2,q^{\ell})$ arising from the set 
of $\Fql$-rational points of the plane curve $\cH$ (see Section 
\ref{s5}). To deal with the parameters of $\cK$, 
the Frobenius nonclassicality property of $\cH$ is not  used. 
In fact, only the degree and the number of $\Fql$-rational points of 
$\cH$ stated in Theorem \ref{herivelto} are used. The approach is the 
natural one: consider $\Fql$-lines 
$\mathcal{L}:\, y+bx+c=0$ and count the number $M_\ell(b,c)$ of 
$\Fql$-rational points of $\cH$ 
lying on $\mathcal{L}$. This 
number is related to the 
degree $d$ of $\cH$ so that $M_\ell(b,c)\leq 
d$. Then $M_\ell(b,c)$ is equal to the number 
of $\Fql$-solutions of the one variable equation
  $$
\bT(x^{q^r+1}+bx+c)=0 
  $$ 
and thus it can be computed by means of the relation 
  \begin{equation}\label{n=qm}
N_\ell(b,c)=qM_\ell(b,c)\, ,
  \end{equation}
where $N_\ell(b,c)$ is the number of $\Fql$-affine points of the 
Artin-Schreier curve of type
  \begin{equation} 
y^q-y=x^{q^r+1}+bx+c\, , \label{AS} 
   \end{equation}
with $r$ defined as in (\ref{r1}). Thus we are led to the problem of 
the computation of rational points on 
curves over finite fields of Artin--Schreier type. Such computations were 
already performed by several authors. For example, in 1989 Wolfmann 
{\cite{wolfmann}} used quadratic forms to calculate the number of 
$\Fql$-affine points of Artin--Schreier curves of type
  $$
y^q-y=ax^s+c\, ,
  $$ 
where $a\in\Fql^*$, $c\in\Fql$, $\ell$ is even and 
$s$ is a certain divisor of $q^\ell-1$. Later on, in 2002, Coulter 
{\cite{c3}} used facts on exponential sums {\cite{c1,c2,c4}} 
to compute the number of $\Fq$-rational points on Artin--Schreier curves 
of type 
    \begin{equation}\label{coulter}
y^{p^n}-y=ax^{p^{\alpha}+1}+L(x)\, , 
     \end{equation}
where $a\in{\Fq}^*$, $t:=\gcd(n,e)$ divides 
$u:=\gcd(\alpha,e)$, with $q=p^e$, 
and $L(x)\in\Fq[x]$ is a $p^t$-linearized 
polynomial. We recall that 
Wolfmann's and Coulter's results have some overlap but they are not 
equivalent.

   \section{The number of rational affine points 
of a class \\
of Artin--Schreier curves}\label{s3}

Throughout this section let $q=p^n$ be a power of 
a prime $p$, and let $\ell$ and $r$ be integers with 
$\ell\geq 2$ and $r\geq 0$. By 
considering the curve (\ref{AS}) and by taking into account 
the type of the curve (\ref{coulter}) studied by Coulter, we are led to 
compute the number of $\Fql$-affine points of Artin--Schreier curves of type 
   $$
y^q-y=ax^{q^r+1}+L(x)+c\, ,
   $$
where $a\in \Fql^*$, $c\in\Fql$ and 
$L(x)=\displaystyle{\sum_{i=0}^{\ell-1}b_ix^{q^i}}\in \Fql[x]$ is a 
$q$-linearized polynomial. If we set  
$b:=\displaystyle{\sum_{i=0}^{\ell-1}b_i^{q^{\ell-i}}}$, 
arguing as in {\cite[Thm. 5.8]{c3}}, then computing $\Fql$-rational 
affine points of curves as 
above is in fact equivalent to computing 
$\Fql$-affine points of Artin--Schreier curves of type
   \begin{equation}\label{beatriz}
y^q-y=ax^{q^r+1}+bx+c\, ,
   \end{equation}
where $a\in\Fql^*$, $b,c\in\Fql$. This observation is  useful 
in computing  the degree of the arcs in 
Section \ref{s5}.

Let $N_{\ell,r}(a,b,c)$ denote the number of $\Fql$-affine points of 
the curve (\ref{beatriz}). By {\cite[Lemma 5.5]{c3}} we have an 
exponential sum of type
   \begin{equation*}
N_{\ell,r}(a,b,c)=\sum\limits_{h\in\Fq}\sum\limits_{x\in\Fql}
\chi_1(hax^{q^r+1}+ hbx+hc)\, ,
   \end{equation*}
where 
$\chi_1(x)={\rm exp}(2\pi\sqrt{-1}{\bf t}(x)/p)$ 
is the canonical additive character of $\Fql$ with 
${\bf t}:\Fql\to\Fp$ being 
the absolute trace map. For $A,B,C\in\Fql$ we 
consider the following Weil sum on $\Fql$:
   $$
R_{\ell,r}(A,B,C):=\sum\limits_{x\in\Fql}\chi_1(Ax^{q^r+1}+Bx+C)\, .
   $$
Thus 

   \begin{equation}\label{N}
N_{\ell,r}(a,b,c)=\sum\limits_{h\in\Fq}R_{\ell,r}(ha,hb,hc)\, .
   \end{equation}
	
It turns out that $R_{\ell,r}(A,B,C)=R_{\ell,r}(A,B,0)\chi_1(C)$, 
where the sum $R_{\ell,r}(A,B,0)$ was computed by 
Coulter {\cite{c1,c2,c4}}. Its computation depends on 
properties of certain polynomials over $\Fql$ such 
as those in Remark \ref{solvable} below 
(see also Remark \ref{solvable2}).
   \begin{remark}\label{solvable} Let $\ell$ and 
$r$ be integers with $\ell\geq 2$ and $r\geq 0$. Suppose that 
$\ell/u$ is odd, where $u=\gcd(\ell,r)$. Let $p$ 
be the characteristic of $\Fql$. Let ${\bT}_u:\Fql\to {\mathbb 
F}_{q^u}$ be the trace map.
   \begin{enumerate}
\item[(1)] If $p=2$, then $\gcd(q^r+1,q^\ell-1)=1$ {\cite[Lemma 2.1]{c4}} 
and thus 
$x^{q^r+1}$ is a permutation polynomial over $\Fql$ and hence over 
$\Fq$. Moreover, for $b\in\Fql$ with ${\bT}_u(b)=1$, the 
equation $x^{q^{2r}}+x+1=b$ has a solution in 
$\Fql$; see the remark after the proof of Theorem 4.2 in {\cite{c4}}.
\item[(2)]If $p>2$, then $f(x)=a^{q^r}x^{q^{2r}}+ax$ with 
$a\in\Fql^*$ is also a permutation polynomial over $\Fql$; 
see the remark after Lemma 2.2 in {\cite{c2}}.
  \end{enumerate}
   \end{remark}
From  Theorems 4.4, 4.5, 4.6 and 4.7 in {\cite{c3}} 
we can compute now the 
sum $R_{\ell,r}(a,b,c)$ as follows.
   \begin{lemma}\label{r2} Let $q=p^n$ be a power of 
a prime $p.$ Let $\ell$ and $r$ be integers with 
$\ell\geq 2$ and $r\geq 0$. Set $u=\gcd(\ell,r).$ Let $a,b,c\in\Fql$, 
$a\neq 0.$ Let $\eta_1$ be the 
quadratic character of $\Fql$ and $\chi_1$ be the canonical additive 
character of $\Fql.$ 
Let $f(x)=a^{q^r}x^{q^{2r}}+ax$ (cf. 
Remark {\rm\ref{solvable}(2)} above). 

{\rm(1)} Let $\ell/u$ be odd. Then
    $$ 
R_{\ell,r}(a,0,c)=
   \begin{cases}
0, &\text{ if }\quad p=2\\
(-1)^{n\ell-1}\, q^{\ell/2}\, \eta_1(a)\ \chi_1(c),&\text{ if }\quad 
p\equiv 1\pmod{4}\\
(-1)^{n\ell-1}\, 
(-1)^{n\ell/2}\, q^{\ell/2}\, \eta_1(a)\, \chi_1(c),&\text{ 
if }\quad p\equiv 3\pmod{4}\, ,
   \end{cases}
    $$
For $b\neq 0,$ the following cases arise.
   \begin{enumerate}
\item[\rm(i)] If $p=2$ and $h\in\Fql,$ then 
  $$
R_{\ell,r}(ah,bh,ch)=R_{\ell,r}(h,b{a_1}^{-1}h,ch)\, ,
  $$
where $a_1\in\Fql^*$ is the solution of $x^{q^r+1}=a.$ 
Moreover$,$ let ${\bT}_u:\Fql\to{\mathbb F}_{q^u}$ be the trace 
map$.$ Then $R_{\ell,r}(1,b,c)=0$ provided that ${\bT}_u(b)\neq 1;$ 
otherwise$,$ there is $w\in\Fql$ such that $b=w^{q^{2r}}+w+1$ and 
  $$
R_{\ell,r}(1,b,c)=\chi_1(w^{q^r+1}+w)\left(\frac{2}{\ell/u}\right)^{nu}
q^{(\ell+u)/2}\chi_1(c)\, ,
   $$ 
where the Jacobi symbol $(\frac{2}{v})$ is defined by the formula
  \begin{equation*}
\left(\frac{2}{v}\right)=
    \begin{cases}
1\, , & \text{ if }\quad v\equiv\pm1\pmod{8}\\
-1\, , & \text{ if }\quad v\equiv\pm3\pmod{8}\, .\\
   \end{cases}
    \end{equation*} 
\item[\rm(ii)] If $p>2,$ then $R_{\ell,r}(a,b,c)$ is given by

	$$
  \begin{cases}
(-1)^{n\ell-1}\, q^{\ell/2}\, \eta_1(-a)\, 
\overline{\chi_1(ax_0^{q^r+1})}\, \chi_1(c)\, , &\text{if\quad 
$p\equiv1\pmod{4}$}\\ 
(-1)^{n\ell-1}i^{3n\ell}\, q^{\ell/2}\, \eta_1(-a)\, 
\overline{\chi_1(ax_0^{q^r+1})}\chi_1(c)\, , &\text{if\quad $p\equiv 
3\pmod{4}$}\, ,\\
  \end{cases}
   $$
where $x_0\in\Fql$ is the solution of $f(x)=-b^{q^r}$ and $i=\sqrt{-1}$.
   \end{enumerate}
{\rm(2)} Let $\ell/u$ be even$.$
   \begin{enumerate}
\item[\rm(i)] If $f(x)=-b^{q^r}$ has no solution in $\Fql,$ 
then $R_{\ell,r}(a,b,c)=0.$
\item[\rm(ii)] If $f(x)$ is a permutation polynomial over $\Fql$ and 
$x_0\in \Fql$ is the solution of $f(x)=-b^{q^r},$ then 
  $$
R_{\ell,r}(a,b,c)=(-1)^{\ell/2u}q^{\ell/2}\, 
\overline{\chi_1(ax_0^{q^r+1})}\, \chi_1(c)\, .
  $$
\item[\rm(iii)] If $f(x)$ is not a permutation polynomial but 
$f(x)=-b^{q^r}$ has a solution $x_0$ in $\Fql,$ then
  $$
R_{\ell,r}(a,b,c)=(-1)^{\ell/2u+1}q^{\ell/2+u}\, 
\overline{\chi_1(ax_0^{q^r+1})}\, \chi_1(c)\, .
  $$
   \end{enumerate}
   \end{lemma} 
Theorems \ref{nabcp2}, \ref{nabcp3} and 
\ref{nabcp4} compute 
$N_{\ell,r}(a,b,c)$. We begin with the case 
$p=2$ and $\ell/u$ odd;
the final result is closely related to 
{\cite[Thm. 6.9]{c3}}.
   \begin{theorem}\label{nabcp2} Let
\begin{enumerate}[\upshape (a)]
	\item $q=2^n$ $;$
\item $\ell$ and $r$ be integers with 
$\ell\geq 2$ and $r\geq 0$ such that $\ell/u$ is odd$,$ where 
$u=\gcd(\ell,r);$ 
\item ${\bT}_u:\Fql\to\mathbb F_{q^u}$ be the trace map$;$
\item $a,b,c\in\Fql,$ $a\neq 0;$ 
\item $a_1\in\Fql^*$ be the solution of $x^{q^r+1}=a.$
   \end{enumerate}
Then $N_{\ell,r}(a,b,c)=N_{\ell,r}(1,b{a_1}^{-1},c).$ If 
${\bT}_u(b)\not\in\Fq^*,$ then $N_{\ell,r}(1,b,c)=q^{\ell};$ 
otherwise$,$ 
   $$
N_{\ell,r}(1,b,c)=q^{\ell}+\chi_1(\omega^{q^r+1}+\omega)
\left(\frac{2}{\ell/u}\right)^{nu}q^{(\ell+u)/2}\, 
\chi_1({\bT}_u(b)^{-2}\, c)\, ,
   $$ 
where $\omega\in\Fql$ is such that $b\,{\bT}_u(b)^{-1}=
\omega^{q^{2r}}+\omega+1,$ 
and $(\frac{2}{v})$ is the Jacobi symbol defined above.
  \end{theorem}
  \begin{proof} We use Lemma \ref{r2}(1)(i). The first part is clear 
from (\ref{N}). Write 
   $$
N_{\ell,r}(1,b,c)=q^\ell+\sum_{h\in{\Fq}^*}R_{\ell,r}(h,hb,hc)\, .
  $$
For each $h\in {\Fq}^*$, $R_{\ell,r}(h,hb,hc)=
R_{\ell,r}(1, h_1b,h_1^2c)$ with 
$h_1\in{\Fq}^*$ such that ${h_1}^{q^r+1}=h$. If 
${\bT}_u(b)\not\in\Fq^*$, 
then ${\bT}_u(h_1b)\neq 1$ and hence $R(h,hb,h,c)=0$ so that 
$N_{\ell,r}(1,b,c)=q^\ell$. Let ${\bT}_u(b)\in{\Fq}^*$; then 
${\bT}_u(h_1b)=1$ if and only if 
$h_1={\bT}_u(b)^{-1}$ so that
  $$
N_{\ell,r}(1,b,c)=q^\ell+R_{\ell,r}(1,h_1b,h_1^2c)
  $$
and the result follows.
    \end{proof}
We recall next some results regarding Gaussian sums over finite fields.
   \begin{lemma}\label{jung} Let $\Fq$ be the 
finite field of 
order $q=p^n$ with $p$ a prime. Let $\eta$ be 
the quadratic character of $\Fq$ 
and let $\chi$ be the canonical additive character of $\Fq.$ For $F\in 
\Fq,$ let $\chi^F(h):=\chi(Fh),$ $h\in\Fq.$ 

\begin{enumerate}[\upshape (i)]
\item Set $\displaystyle{G(\eta,\chi^F):=\sum_{h\in\Fq^*}\eta(h)\chi^F(h)}$, then

$$G(\eta,\chi^F)=\begin{cases}
   0, & \text{if $F=0$}\\
   (-1)^{n-1}q^{1/2}\, \eta(F), & \text{if $F\neq0$, $p\equiv 
1\pmod{4}$}\\
   (-1)^{n-1}(-1)^{n/2}q^{1/2}\, \eta(F), & \text{if $F\neq0$, $p\equiv 
3\pmod{4}$.}
     \end{cases}
    $$
\item 
    $$
   G(1,\chi^F):=\sum_{h\in\Fq^*}\chi^F(h)=\begin{cases}
q-1,&\text{if $F=0$}\\
-1,&\text{if $F\neq 0$.}
  \end{cases}
   $$
	\end{enumerate}
   \end{lemma}
  \begin{proof} (i) If $F=0$, see {\cite[Thm. 5.4]{nl}}. If $F\neq 0$, the 
result follows from Theorems 5.12(i), 5.15 in {\cite{nl}}.

(ii) It follows from relation (5.8) in {\cite[p. 192]{nl}}.
   \end{proof}
	
The following result is closely related to 
{\cite[Thm. 6.10]{c3}}.
   \begin{theorem}\label{nabcp3} Let
\begin{enumerate}[\upshape (a)]
\item $q=p^n$ be a power of a prime $p>2;$ 
\item $\ell$ and $r$ be integers with 
$\ell\geq 2$ and $r\geq 0$ such that $\ell/u$ is odd$,$ where 
$u=\gcd(\ell,r);$ 
\item $a,b,c\in \Fql,\ a\neq 0;$ 
\item ${\bT}:\Fql\to\Fq$ be the trace map$;$ 
\item $\eta_1$ be the quadratic character of $\Fql;$ 
\item $f(x)=a^{q^r}x^{q^{2r}}+ax$ and $x_0\in\Fql$ be the 
solution of $f(x)=-b^{q^r};$  
\item $c_1=ax_0^{q^r+1}-c.$
    \end{enumerate}

There are two cases:

\begin{enumerate}[\upshape (1)]
\item Let $\ell$ be odd. If ${\bT}(c_1)=0,$ then 
$N_{\ell,r}(a,b,c)=q^\ell;$ otherwise, $N_{\ell,r}(a,b,c)$ is given by
    \begin{equation*}
q^\ell+
   \begin{cases}
q^{(\ell+1)/2}\, \eta_1(a\, {\bT}(c_1)), &\text{if $p\equiv 
1\pmod{4}$}\\
(-1)^{n(\ell+1)/2}\, q^{(\ell+1)/2}\, \eta_1(a\, {\bT}(c_1)), 
&\text{if $p\equiv 3\pmod{4}$,\, $b=0$} \\
(-1)^{n(3\ell+1)/2}\, q^{(\ell+1)/2}\, \eta_1(a\, {\bT}(c_1)), 
&\text{if $p\equiv 3\pmod{4}$,\, $b\neq 0$}\, .
   \end{cases}
  \end{equation*}
\item Let $\ell$ be even$.$ Then $N_{\ell,r}(a,b,c)$ is given by
   \begin{equation*}
 q^\ell+
   \begin{cases}
(-1)q^{\ell/2}(q-1)\, \eta_1(a), &\text{if $p\equiv 1\pmod{4}$,\,  
${\bT}(c_1)=0$}\\
q^{\ell/2}\, \eta_1(a), &\text{if $p\equiv 1\pmod{4}$,\,  
${\bT}(c_1)\neq 0$}\\
(-1)^{1+n\ell/2}q^{\ell/2}(q-1)\, \eta_1(a), &\text{if  
$p\equiv3\pmod{4}$,\,  ${\bT}(c_1)=0$}\\
(-1)^{n\ell/2}q^{\ell/2}\, \eta_1(a), &\text{if  
$p\equiv3\pmod{4}$,\, ${\bT}(c_1)\neq0$\, .}
   \end{cases}
   \end{equation*}
	\end{enumerate}
    \end{theorem}
    \begin{proof} Let $\eta$ be the quadratic 
character of $\Fq$ and 
$\chi$ be the canonical additive character of $\Fq$. By the transitivity 
property of trace maps, $\chi_1=\chi\circ \bT$. Concerning 
quadratic characters of $\Fql$ and $\Fq$, for $h\in \Fq^*$ we have 
$\eta_1(h)=\eta(h)$ if $\ell$ is odd; otherwise $\eta_1(h)=1$.

Let $\ell$ be odd. Let $p\equiv 1\pmod{4}$. Then, from (\ref{N}) and Lemma
\ref{r2}(1), $$ N_{\ell,r}(a,0,c)=q^\ell+(-1)^{n\ell-1}q^{\ell/2}\eta_1(a)
G(\eta,\chi^F)\, , $$ where $G(\eta,\chi^F)$ is the Gaussian sum in Lemma
\ref{jung} with $F={\bT(c_1)}$. Now the result follows. The case
$p\equiv 3\pmod{4}$ is similar.

Let $\ell$ be even. We use the Gaussian sum $G(1,\chi^F)$ in Lemma 
\ref{jung} and the result follows.
    \end{proof}
The following result is close to {\cite[Thm. 7.11]{c3}}.
    \begin{theorem}\label{nabcp4} Let 
\begin{enumerate}[\upshape (a)]
\item $q=p^n$ be the power of a prime $p;$ 
\item $\ell$ and $r$ be integers with $\ell\geq 2$ and $r\geq 0$ 
such that $\ell/u$ is even with $u=\gcd(\ell,r);$
\item ${\bT}:\Fql\to\Fq$ be the trace map$;$ 
\item $a,b,c\in\Fql$, $a\neq 0;$ 
\item $f(x)=a^{q^r}x^{q^{2r}}+ax.$ 
  \end{enumerate}
Suppose that $f(x)=-b^{q^r}$ has no roots in $\Fql.$ Then 
$N_{\ell,r}(a,b,c)=
q^{\ell};$ otherwise$,$ let 
$x_0\in\Fql$ be a root of $f(x)=-b^{q^r}.$ Set $c_1=ax_0^{q^{r}+1}-c.$ 

\begin{enumerate}[\upshape (1)]
\item If $f(x)$ is a permutation polynomial over $\Fql,$ then
   \begin{equation*}
N_{\ell,r}(a,b,c)=q^\ell+
  \begin{cases}
(-1)^{\ell/2u}q^{\ell/2}(q-1), & \text{if $\bT(c_1)=0$}\\
(-1)^{\ell/2u+1}q^{\ell/2}, & \text{if $\bT(c_1)\neq0$}\, .
    \end{cases}
   \end{equation*}
\item If $f(x)$ is not a permutation polynomial, then
  \begin{equation*}
N_{\ell,r}(a,b,c)=q^{\ell}+
  \begin{cases}
(-1)^{\ell/2u+1}q^{\ell/2+u}(q-1), & \text{ if  $\bT(c_1)=0$}\\
(-1)^{\ell/2u}q^{\ell/2+u}, & \text{ if $\bT(c_1)\neq 0$}\, .
  \end{cases}
  \end{equation*}
	\end{enumerate}
  \end{theorem}
    \begin{proof} The first part follows from (\ref{N}) and Lemma 
\ref{r2}(2)(i). If $f(x)$ is a permutation polynomial over $\Fql$ with 
$x_0$ as above, by (\ref{N}) and Lemma \ref{r2}(2)(ii) we have
  $$
N_{\ell,r}(a,b,c)=q^\ell+(-1)^{\ell/2u}q^{\ell/2}G(1,\chi^F)
  $$
with $F=\bT(c_1)$ and the result follows from Lemma \ref{jung}. If 
$f(x)$ is not a permutation polynomial, 
the proof is similar.
  \end{proof} 

    \section{On maximal Artin--Schreier 
curves}\label{sectionmax}

Let $q=p^n$ be a power of a prime $p$ and let $\ell$ and $r$ be integers 
with $\ell\geq 2$ and $r\geq 0$. Let $a,b,c\in\Fql$, $a\neq 0$. In 
Section \ref{s3} we computed the number $N_{\ell,r}(a,b,c)$ of 
$\Fql$-rational affine points of the Artin--Schreier curve of type 
(\ref{beatriz}), namely
   $$
y^q-y=ax^{q^r+1}+bx+c\, .
   $$
This curve has exactly one singular point which is unibranched; thus the 
number of $\Fql$-rational points of its 
nonsingular model over $\Fql$, denoted by $\cX=\cX_{\ell,r}(a,b,c)$, is 
exactly
  $$
N_{\ell,r}(a,b,c)+1\, .
  $$
The Hasse--Weil bound (see  {\cite[Thm. 
V.2.3]{sti}}, {\cite[Thm. 9.18]{HKT}}) asserts that
  $$
|\#\cX(\Fql)-(q^\ell+1)|\leq 2gq^{\ell/2}\, ,
  $$
where $g$ is the genus of the curve. Here we have $g=q^r(q-1)/2$; see, 
for example, {\cite[Prop. VI.4.1]{sti}}). We are looking for examples of {\em 
$\Fql$-maximal curves} of type 
$\cX_{\ell,r}(a,b,c)$, that is, those  whose number of 
$\Fql$-rational points attains the Hasse--Weil upper bound; equivalently,
 those curves
such that 
   \begin{equation}\label{hw}
N_{\ell,r}(a,b,c)=q^\ell+q^{\ell/2+r}(q-1)\, .
   \end{equation}
It follows then that $q^{\ell/2}$ must be an integer, that is, 
 $n\ell$ must be an even 
integer. See {\cite[Ch. 10]{HKT}} for general results on maximal curves. 

We consider two cases according to the parity of $\ell/u$ with
$u:=\gcd(\ell,r)$.

{\bf Case A:} Suppose that $\ell/u$ is odd. If 
$p=2$, Theorem \ref{nabcp2} does not provide an example where (\ref{hw}) 
holds true. Let $p>2$. Let $f(x)$ and $c_1$ be as in Theorem \ref{nabcp3} 
(cf. Remark \ref{solvable}). If 
$\cX_{\ell,r}(a,b,c)$ is $\Fql$-maximal, Theorem \ref{nabcp3} implies 
that $\ell$ must be 
even, $r=0$ and $\bT(c_1)=0$. Under these 
conditions, the curve 
$\cX_{\ell,0}(a,b,c)$ is 
$\Fql$-maximal if and only if either $p\equiv 1\pmod{4}$ and $a$ is not 
a square in 
$\Fql^*$; or $p\equiv 3\pmod{4}$, $a\in \Fql^*$ is a square 
and $n\ell/2$ is odd; or $p\equiv 3\pmod{4}$, 
$a\in \Fql^*$ is not a 
square and $n\ell/2$ is even.

{\bf Case B:} Suppose that $\ell/u$ is even. Thus 
$r\geq 1$. By Theorem \ref{nabcp4} a 
necessary condition to have (\ref{hw}) is that $f(x)=-b^{q^r}$ has a 
root in $\Fql$ and $\bT(c_1)=0$, where $f(x)$ and $c_1$ are as in Case A 
above. Under these conditions, $\cX_{\ell,r}(a,b,c)$ is 
$\Fql$-maximal if and only if $u=\gcd(\ell,r)=r$ 
  and $\ell/(2u)$ is odd.

We summarize the above computations in the following.
   \begin{theorem}\label{max} Let 
   \begin{enumerate}
\item[\rm(a)] $q=p^n$ be a power of a prime $p;$ 
\item[\rm(b)] $\ell$ and $r$ be integers with $\ell\geq 2$ and $r\geq 
0;$ 
\item[\rm(c)] $\bT:\Fql\to \Fq$ be the trace map$;$ 
\item[\rm(d)] $a,b,c\in\Fql$, $a\neq 0;$
\item[\rm(e)] $f(x)=a^{q^r}x^{q^{2r}}+ax.$ 
   \end{enumerate}
Let $\cX=\cX_{\ell,r}(a,b,c)$ be the 
nonsingular model of the Artin--Schreier 
curve of type {\rm(\ref{beatriz})} over $\Fql.$ If $\cX$ is a 
$\Fql$-maximal curve$,$ then the following conditions must be satisfied$:$
  \begin{enumerate}[\rm(i)]
\item $n\ell$ is even$;$
\item The equation $f(x)=-b^{q^r}$ has a solution $x_0\in\Fql$ 
such that $\bT(c_1)=0,$ where $c_1=a{x_0}^{q^r+1}-c.$
  \end{enumerate}
Conversely$,$ if these conditions are satisfied then $\cX$ is 
$\Fql$-maximal if and only if one of the 
following conditions hold true$:$
   \begin{enumerate}[\rm(1)]
\item $r=0,$ $\ell$ even$,$ $p\equiv 1\pmod{4},$ and $a$ is not a 
square in $\Fql;$
\item $r=0,$ $\ell$ even$,$ $p\equiv 3\pmod{4},$ and either 
$n\ell/2$ odd and $a$ is a square in $\Fql,$ or $n\ell/2$ even and $a$ 
is not a square in $\Fql;$
\item $r\geq 1,$ $2r$ divides $\ell$ such that $\ell/2r$ is odd, 
$f(x)$ is not a permutation polynomial over $\Fql.$
   \end{enumerate}
   \end{theorem}
	
   \begin{remark}\label{rem4.1} Under conditions (i), (ii) of Theorem \ref{max}, the curve $\cX_{\ell,r}(a,b,c)$ is 
$\Fql$-minimal (in the sense that the lower Hasse-Weil bound above is attained) if and only if one of the following conditions hold 
true:

   \begin{enumerate}[$(1')$]
	
\item $r=0,$ $\ell$ even$,$ $p\equiv 1\pmod{4},$ and $a$ is a 
square in $\Fql;$
\item $r=0,$ $\ell$ even$,$ $p\equiv 3\pmod{4},$ and either 
$n\ell/2$ odd and $a$ is not a square in $\Fql,$ or $n\ell/2$ even and 
$a$ is a square in $\Fql;$
\item $r\geq 1,$ $2r$ divides $\ell$ such that $\ell/2r$ is even, 
$f(x)$ is not a permutation polynomial over $\Fql.$
   \end{enumerate}
   \end{remark}
   \begin{remark}\label{rem4.11} There are examples of maximal curves 
for each case in Theorem \ref{max}; cf. {\cite[Thm. 3.3]{c3}}, 
{\cite[Thm. 1]{saeed}}.
   \end{remark}
   \begin{remark}\label{rem4.2} The notation is as in Theorem \ref{max}. Let 
$p\equiv 3\pmod{4}$ and $n\ell/2$ be odd. If $a$ is a nonzero square, 
then $\cX=\cX_{\ell,0}(a,b,c)$ can be defined by an equation of type
   $$
y^q-y=x^2+c'\, ,
   $$
where $c'\in\Fql$. Since the solution in $\Fql$ of $f(x)=2x=0$ is 
$x_0=0$, by Theorem \ref{max} $\bT(c')=0$ and thus $\cX$ is 
uniquely defined by 
  $$
y^q-y=x^2\, .
  $$
This example is missing in {\cite[Thm. 8.12]{c3}} and it is a particular 
case in {\cite[Thm. 1]{saeed}}. 
   \end{remark}
   \begin{remark}\label{rem4.3} Let $\cX=\cX_{\ell,r}(a,b,c)$ be a 
$\Fql$-maximal 
curve satisfying Theorem \ref{max}(3) with $\ell=2r$. Then by {\cite[Thm. 
2.3]{garcia-ozbudak}} $\cX$ is $\Fql$-isomorphic to a curve of type 
$\cX_{\ell,r}(\alpha,0,0)$, where $f(x)=\alpha^{q^r}x^{q^{2r}}+\alpha 
x$ is not a permutation polynomial. We may choose $\alpha=1$ if $p=2$; 
otherwise 
$\alpha=\varsigma^{(q^r+1)/2}$, $\varsigma$ a generator of $\Fql^*$ 
{\cite[Prop. 3.2]{c3}}.
   \end{remark}
   \begin{remark}\label{rem4.4} \c{C}ak\c{c}ak and \"Ozbudak {\cite{CO}} 
considered maximal curves that include those studied by 
Coulter {\cite{c3}}; in particular, they show that these examples are 
covered by 
Hermitian curves. As a matter of fact, there are maximal curves which 
are not 
covered by Hermitian curves; cf. {\cite{GK}}. Are maximal curves in 
(\ref{beatriz}) with $\bT(c)=0$  isomorphic to Coulter's 
curves? Is a maximal curve in (\ref{beatriz}) 
with $\bT(c)\neq 0$ isomorphic to a curve in {\cite{CO}}? Must such 
a curve  be covered by the Hermitian curve?
   \end{remark}

    \section{The arc arising from $\cH$}\label{s5}

Throughout this section we let $q=p^n$ be a power of a
prime $p$, $\ell$ an integer with $\ell\geq 3$, and 
$r=r(\ell)$ be the integer defined in (\ref{r1}); in particular, 
$u=\gcd(\ell,r)=1$. We are interested in the arc property derived from 
the pointset  
  $$
\cK=\cH(\Fql)\subseteq PG(2,q^\ell)
  $$ 
defined from the set of $\Fql$-rational points of the curve $\cH$ 
introduced in Section \ref{s2}. By Theorem \ref{herivelto}, $\cK$ is  an 
$(N,d)$-arc with parameters
   \begin{equation}\label{eq5.1}
N=q^{2\ell-1}+1\qquad {\rm and}\qquad  
d=q^{\ell-1}+q^{r-1}\, .
   \end{equation}
By (\ref{n=qm}), the 
degree $d$ of the arc is also closely related to the 
number $N_\ell(b,c):=N_{\ell,r}(1,b,c)$ of 
$\Fql$-affine points of Artin--Schreier curves of type (\ref{AS}), 
namely
   $$
y^q-y=x^{q^r+1}+bx+c\, ,
   $$
where $b,c\in \Fql$. We have $N_\ell(b,c)\leq qd$. The numbers 
$N_\ell(b,c)$ can be deduced directly from 
Theorems \ref{nabcp2}, \ref{nabcp3}, \ref{nabcp4} above. For the sake 
of convenience we explicitly state such computations below.
   \begin{lemma}\label{lemma5.1} Consider the same notation as 
above$;$ in particular$,$ 
$q=p^n$ with $p$ a prime and $\ell$ is an integer with $\ell\geq 3$, 
$b,c\in\Fql$. In 
addition$,$ let $\bT:\Fql\to\Fq$ be the 
trace map and let $\chi_1$ be the canonical additive character of 
$\Fql$. Let $f(x)=x^{q^{2r}}+x$ with $r$ as in {\rm(\ref{r1})}.

{\rm(1)} Suppose that $\ell$ is odd$.$ 
   \begin{enumerate}
   \item[\rm(i)] Let $p=2$. If 
$\bT(b)=0,$ then $N_{\ell}(b,c)=q^{\ell};$ 
otherwise$,$ 
   $$
N_{\ell}(b,c)=q^{\ell}+\chi_1(\omega^{q^r+1}+\omega)
\left(\frac{2}{\ell}\right)^{n}q^{(\ell+1)/2}\, \chi_1(\bT(b)^{-2}\, 
c)\, ,
   $$ 
where $\omega\in\Fql$ is such that 
$b\,\bT(b)^{-1}=\omega^{q^{2r}}+\omega+1,$ 
and $(\frac{2}{v})$ is the Jacobi symbol.
  \item[\rm(ii)] Let $p>2$. Let $x_0$ be 
the solution of 
$f(x)=-b^{q^r}$ (cf. Remark \ref{solvable}). Let $\eta$ be the quadratic 
character of 
$\Fq$. Set $c_1=ax_0^{q^r+1}-c.$ If $\bT(c_1)=0,$ then 
$N_{\ell}(b,c)=q^\ell;$ otherwise$,$
    \begin{equation*}
N_{\ell}(b,c)=q^\ell+
   \begin{cases}
q^{(\ell+1)/2}\, \eta(\bT(c_1)), &\text{if $p\equiv 1\pmod{4}$}\\
(-1)^{n(\ell+1)/2}\, q^{(\ell+1)/2}\, \eta(\bT(c_1)), 
&\text{if $p\equiv 3\pmod{4}$,\, $b=0$} \\
(-1)^{n(3\ell+1)/2}\, q^{(\ell+1)/2}\, \eta(\bT(c_1)), 
&\text{if $p\equiv 3\pmod{4}$,\, $b\neq 0$}\, .
   \end{cases}
  \end{equation*}
  \end{enumerate}
{\rm(2)} Suppose that $\ell$ is even$.$ If 
$f(x)=-b^{q^r}$ has no roots in $\Fql,$ then 
$N_{\ell}(b,c)=q^{\ell};$ otherwise$,$ let 
$x_0\in\Fql$ be a root of $f(x)=-b^{q^r}.$ Set $c_1=ax_0^{q^{r}+1}-c.$ 
   \begin{enumerate}
\item[(i)] If $f(x)$ is a permutation polynomial over $\Fql,$ then
   \begin{equation*}
N_{\ell}(b,c)=q^\ell+
  \begin{cases}
(-1)^{\ell/2}q^{\ell/2}(q-1), & \text{if $\bT(c_1)=0$}\\
(-1)^{\ell/2+1}q^{\ell/2}, & \text{if $\bT(c_1)\neq0$}\, .
    \end{cases}
   \end{equation*}
\item[(ii)] If $f(x)$ is not a permutation polynomial, then
  \begin{equation*}
N_{\ell}(b,c)=q^{\ell}+
  \begin{cases}
(-1)^{\ell/2+1}q^{\ell/2+1}(q-1), & \text{ if  $\bT(c_1)=0$}\\
(-1)^{\ell/2}q^{\ell/2+1}, & \text{ if $\bT(c_1)\neq 0$}\, .
  \end{cases}
  \end{equation*}
  \end{enumerate}
  \end{lemma}
Next we are concerned with the permutation property of the polynomial $
f(x)$ which arises in the lemma above.
   \begin{remark}\label{solvable2} Let 
$f(x)=x^{q^{2r}}+x\in\Fql[x]$ 
with $q$ a power of a prime $p$, $\ell$ an integer with $\ell\geq 3$, 
and $r$ as in (\ref{r1}). If $p=2$, it is clear that $f(x)$ 
is not a permutation polynomial. If $p>2$, then 
Remark \ref{solvable} can be improved so that 
$f(x)$ is a permutation 
polynomial if and only if either $\ell$ is odd, or 
$\ell\equiv 2\pmod{4}$; see the remark after the proof of Theorem 4.1 in 
{\cite{c1}}.
   \end{remark}
Recall that $N$, $d$ and $r$ stand for the 
integers defined in (\ref{eq5.1}) 
and (\ref{r1}). 
  \begin{question}\label{question5.1} Is the pointset 
$\cK=\cH(\Fql)$ defined above a complete $(N,d)$-arc in $PG(2,q^\ell)$?
  \end{question}
{\bf Case A:} The answer to Question \ref{question5.1} is affirmative 
provided that $p\equiv 1\pmod{4}$ and $\ell$ is odd with $\ell\geq 3$.

In fact, let $P\in PG(2,q^\ell)\setminus\cK$. We shall show that 
there is a line $\cL:y+bx+c=0$ in $PG(2,q^\ell)$ 
such that $P\in\cL$ and $\#\cK\cap\cL=d$. 
If $P=(A:B:1)$, we look for $\cL$ with $c=-bA-B$ (so that $P\in\cL$). 
Let us consider the Artin--Schreier curve of type
   \begin{equation*}
y^q-y=x^{q^r+1}-Ax^{q^r}-Ax^{q^{r-1}}+B-\lambda\, ,
   \end{equation*}
where $\lambda\in\Fql$ is such that $\bT(\lambda)$ is a 
nonzero square in $\Fq$. As already mentioned in Section \ref{s3} 
(cf. {\cite[Thm. 5.8]{c3}}), 
this curve has the same number of $\Fql$-affine points as a certain 
curve of type (\ref{beatriz}). Thus, by Lemma \ref{lemma5.1}(1)(ii),
the curve above has 
at least $q^\ell-q^r\ \Fql$-affine points; let $(x_0,y_0)$ 
be one of such points and set $b:=-{x_0}^{q^r}-{x_0}^{q^{r-1}}$. Then
  $$
-b^{q^r}={x_0}^{q^{2r}}+{x_0}^{q^{2r-1}}\, ,
  $$
and thus $x_0$ is also the 
solution of the equation $f(x)=-b^{q^r}$, with 
$f(x)=x^{q^{2r}}+x$, as $2r-1=\ell$. Moreover, by 
construction,
  $$
c_1={x_0}^{q^r+1}-c=
{x_0}^{q^r+1}+bA+B=
{x_0}^{q^r+1}-A{x_0}^{q^r}-A{x_0}^{q^{r-1}}+B;
  $$
so  $\bT(c_1)=\bT(\lambda)$ is a 
nonzero square in $\Fq$. 
The result follows from Lemma \ref{lemma5.1}(1)(ii) 
and (\ref{n=qm}). 
Now let $P=(1:B:0)$. Here we look for a line of 
type $\cL: y-Bx+c=0$ 
with some $c\in\Fql$. Let $x_0\in\Fql$ be a 
solution of $f(x)=B^{q^r}$ 
(cf. Remark \ref{solvable2}) and let $c$ be such 
that $\bT({x_0}^{q^r+1}-c)$ is a nonzero square 
in $\Fq$; the result follows.

{\bf Case B:} The answer to Question 
\ref{question5.1} is also affirmative if 
$p\equiv 3\pmod{4}$ and $\ell$ is odd with 
$\ell\geq 3$. 
The proof is similar to Case A and here 
we choose $\lambda\in\Fql$ according to the 
parity of either $n(\ell+1)/2$ or $n(3\ell+1)/2$.

{\bf Case C:} Let $p>2$ and $\ell$ be even with 
$\ell\geq 6$ and 
$\ell\equiv 2\pmod{4}$. Then the answer to 
Question \ref{question5.1} is negative.

In fact, here $\cK$ is a complete 
$(N,d_1)$-arc with 
$d_1=q^{\ell-1}+q^{r-3}$ (which is clearly less 
than the degree $d$ of $\cH$). 
To see this, let $\cL$ be a line in $PG(2,q^\ell)$ 
defined by the 
equation $\alpha X+\beta Y+\gamma Z=0$. We claim 
that $\#\cK\cap\cL\leq 
d_1$. If $\beta=0$, then it is easy to see that 
$\#\cK\cap\cL\leq 
q^{\ell-1}$. For $\beta\neq 0$, the claim 
follows from Lemma \ref{lemma5.1}(2)(i) 
as $\ell/2=r-2$ and $f(x)$ is a permutation 
polynomial (see Remark \ref{solvable2}).

Now we prove the completeness of the 
$(N,d_1)$-arc $\cK$. The proof is similar to 
Case A. Let $P\in 
PG(2,q^\ell)\setminus 
\cK$. If $P=(A:B:1)$, we look for a line 
$\cL:\:y+bx+c=0$ such that  
$c=-bA-B$ and $\#\cK\cap\cL=d_1$. Let us consider 
the Artin--Schreier curve of type
   $$
y^q-y=x^{q^r+1}-Ax^{q^r}-Ax^{q^{r-4}}+B-\lambda\, ,
   $$
where $\lambda\in\Fql$ is such that 
$\bT(\lambda)\neq0$. We see that this 
curve has at least 
$q^{\ell}-q^{r-2}(q-1)$ $\Fql$-affine points. 
Let $(x_0,y_0)$ be one of 
these points, and let $b:=-x_0^{q^r}-x_0^{q^{r-4}}$. 
Therefore
$f(x_0)=-b^{q^r}$ since $2r-4=\ell$. Also, 
by construction, 
$\bT(x_0^{q^r+1}-c)=\bT(\lambda)\neq0$. 
Now the 
result follows from Lemma \ref{lemma5.1}(2)(i) 
and (\ref{n=qm}). Finally, 
let $P=(1:B:0)$ and $x_0\in\Fql$ be a 
solution of $f(x)=B^{q^r}$ (cf. Remark 
\ref{solvable2}); choose $c\in\Fql$ such that 
$\bT(x_0^{q^r+1}-c)\neq 0$. Then the 
line $\cL:y-Bx+c=0$ is such that $P\in\cL$ and 
$\#\cK\cap\cL=d_1$ by Lemma 
\ref{lemma5.1}(2)(i).

{\bf Case $\mathbf{D_1}$:} Let $p=2$ and $\ell$ be odd 
with $\ell\geq 3$. We assume 
$q=2^n$ with $n$ even; otherwise, we assume $n$ odd 
and $\ell\equiv\pm 1\pmod{8}$. Here the answer to Question \ref{question5.1} 
is also negative.

In fact, let us consider the following set:
  $$
\bar\cK:=\{(1:B:0)\in PG(2,q^\ell): \bT(B)=0\}\, .
  $$
We claim that the pointset 
  $$
\cK_1:=\cK\cup \bar\cK
  $$ 
is a complete $(N_1,d)$-arc in 
$PG(2,q^\ell)$ with $N_1=N+q^{\ell-1}$. That 
$\cK_1$ is an $(N_1,d)$-arc is clear by 
Lemma \ref{lemma5.1}(1)(i); next we prove 
its completeness. Let $P\in 
PG(2, \Fql)\setminus \cK_1$.

If $P=(A:B:1)$, we look for a line 
$\cL:\:y+bx+c=0$ with $c=-bA-B$ such 
that $\#\cK_1\cap\cL=d$. 
Let $\gamma \in \Fq^*$ and consider the 
Artin--Schreier curve of type
  $$
y^q-y=x^{q^r+1}+x-
(x^q+x+1)A\gamma^{-1}-B\gamma^{-2}\, .
  $$
Arguing as in Case A, we can see that this curve 
has at least one affine 
$\Fql$-point, say $(x_0,y_0)$. We let 
$b:=(x_0^{q^{2r}}+x_0+1)\gamma$. Then, 
as $2r=\ell+1$ and $p=2$, $\bT(b)=\gamma$ so that 
$b{\bT(b)}^{-1}=x_0^{q^{2r}}+x_0+1$. After some 
computation,
    $$
x_0^{q^r+1}+x_0+{\bT(b)}^{-2}c=y_0^q-y_0
    $$
and, by the transitivity of the trace map,
   $$
N_\ell(b,c)=q^\ell+\left(\frac{2}{\ell}\right)^{n}q^r
   $$
by Lemma \ref{lemma5.1}(1)(i); 
the result follows. Now let 
$P=(1:B:0)$ 
with $\bT(B)\neq0$. We look for a line $\cL: 
y-Bx+c=0$ with $\#\cK_1\cap\cL=d$. 
Let $\omega \in\Fql$ be such that 
$B\,{\bT(B)}^{-1}=\omega^{q^{2r}}+\omega+1$ 
(see Remark \ref{solvable}). 
Define $c=(\omega^{q^r+1}+\omega){\bT(B)}^2$. 
Then 
   $$
\omega^{q^r+1}+\omega+c{\bT(B)}^{-2}=0\, ,
  $$
and the result follows again from Lemma 
\ref{lemma5.1}(1)(i). 

{\bf Case $\mathbf{D_2}$:} Let $p=2$ and $\ell$ be odd 
with $\ell\geq 3$. We assume 
$q=2^n$ with $n$ odd and $\ell\equiv\pm 
3\pmod{8}$. Here the answer to Question \ref{question5.1} 
is also negative. 

In fact, let us consider the set
  $$
  \bar\cK:=\{(1:B:0)\in PG(2,q^\ell): \bT(B)\neq 0\}\, .
  $$
We claim that the pointset 
  $$
\cK_1:=\cK\cup \bar\cK
  $$ 
is in fact a complete $(N_1,q^{\ell-1})$-arc in 
$PG(2,q^\ell)$ with $N_1=N+q^\ell-q^{\ell-1}$. 
That $\cK_1$ is an $(N_1,q^{\ell-1})$-arc is clear. To see its completeness, let $P\in 
PG(2,q^\ell)\setminus \cK_1$. Let $P=(A:B:1)$ 
and let $\cL$ be the line $y+bx+c=0$ with 
$c=-bA-B$ so that $P\in \cL$; if we let $\bT(b)=0$, then 
$\#\cK_1\cap\cL=q^{\ell-1}$ by 
Lemma \ref{lemma5.1}(1)(i). Now let $P=(1:B:0)$ 
with $\bT(B)=0$; here we let $\cL$ be the 
line $y-Bx=0$ and the result follows by Lemma 
\ref{lemma5.1}(1)(i) again.

{\bf Case E:} Let $p\geq 2$ be a prime and $\ell$ be 
even with $\ell\geq 4$ and $\ell\equiv 0\pmod{4}$. Here the answer to Question \ref{question5.1} is 
also negative.

In fact, set $f(x)=x^{q^{2r}}+x$ and let $H$ be the set of elements $B\in\Fql$ 
such that 
the equation $f(x)=B^{q^r}$ has a solution in 
$\Fql$. Let us fix a set 
$H_1\subseteq \Fql\setminus H$ with 
$\# H_1=q^{\ell-1}+q^{r-1}-1$; this selection of 
$H_1$ is possible since $\#H\leq q^{\ell-2}$. 
Then the pointset
  $$
\cK_2:=\cK\cup\{(1:B:0)\in PG(2,q^\ell):B\in H_1\}
  $$ 
is a complete $(N_2,d)$-arc, with 
$N_2=N+\# H_1=q^{2\ell-1}+q^{\ell-1}+q^{r-1}$. 

Arguing as in Case C, it is easy to see 
that $\cK_2$ is in fact an 
$(N_2,d)$-arc. To derive its completeness, let $P\in 
PG(2,q^\ell)\setminus \cK_2$. If $P=(A:B:1)$, we 
proceed as in Case C by 
means of Remark \ref{solvable2} and 
Lemma \ref{lemma5.1}(2)(ii). Let now 
$P=(1:B:0)$ with $B\in H$, and $x_0\in\Fql$ a 
solution for 
$f(x)=B^{q^r}$. Let $c\in\Fql$ such that 
$\bT(x_0^{q^r+1}-c)\neq0$ 
and consider the line $y-Bx+c=0$; the result 
follows.

{\bf Case F:} Let $p=2$ and $\ell$ be even with 
$\ell\geq 4$ and $\ell\equiv 2\pmod{4}$. In this case, the answer to Question 
\ref{question5.1} is also negative.

In fact, let $H$ be the set defined in Case E and 
let us fix 
a set $H_2\subseteq \Fql\setminus H$ such that 
$\# H_2=q^{\ell-1}+q^{r-2}(q-1)-1$. Then the pointset 
   $$ 
\cK_3:=\cK\cup\{(1:B:0):B\in H_2\}
   $$ 
is a complete $(N_3,d_2)$-arc with 
\[
N_3=N+\# H_2=q^{2\ell-1}+q^{\ell-1}+q^{r-2}(q-1),\quad 
d_2=q^{\ell-1}+q^{r-2}(q-1).
\]
 The proof of this case is analogous to 
Case E by using Lemma \ref{lemma5.1}(2)(ii) once 
again.

We summarize the above computations in the following.
  \begin{theorem}\label{main-result} Let $\cH$ 
be the plane curve over $\Fql$ defined in Section 
\ref{s2}, where $q=p^n$ is a power of a prime 
$p\geq 2$ and $\ell$ 
is an integer with 
$\ell\geq 3$. 
Let $\cK=\cH(\Fql)\subseteq PG(2,q^\ell)$ be the 
set of $\Fql$-rational 
points of $\cH$. Let $N=\#\cH(\Fql)=q^{2\ell-1}+1$ 
and $d=q^{\ell-1}+q^{r-1}$ be the number of 
$\Fql$-rational points and the 
degree of $\cH,$  where $r$ is the 
integer defined in {\rm(\ref{r1})}.
   \begin{enumerate}
   \item[\rm(1)] If $p>2$ and $\ell$ is odd$,$ 
then $\cK$ 
is a complete $(N,d)$-arc in $PG(2,q^\ell);$
   \item[\rm(2)] If $p>2$ and $\ell$ is even 
with 
$\ell\equiv 2\pmod{4},$ then 
$\cK$ is a complete $(N,d_1)$-arc in $PG(2,q^\ell)$ 
with $d_1=q^{\ell-1}+q^{r-3};$
  \item[\rm(3)] Let $p=2$ and $\ell$ be odd$.$ 
Suppose that $n$ is even or 
$\ell\equiv \pm 1\pmod{8}$. Let us define the 
set $
\bar\cK:=\{(1:B:0)\in 
PG(2,q^\ell): \bT(B)=0\},$ 
being $\bT:\Fql\to\Fq$ the trace 
map. Then the pointset 
  $$
\cK_1:=\cK\cup\bar\cK
  $$ 
is a complete $(N_1,d)$-arc$,$ with 
$N_1=N+q^{\ell-1};$
    \item[\rm(4)] Let $p=2$ and $\ell$ be 
odd$.$ Suppose that $n$ is odd and 
$\ell\equiv \pm 3\pmod{8}$. Let us define the set
   $
\bar\cK:=\{(1:B:0)\in PG(2,q^\ell): \bT(B)
\neq 0\},$ 
being $\bT:\Fql\to\Fq$ the trace 
map. Then the pointset 
  $$
\cK_1:=\cK\cup\bar\cK
  $$ 
is a complete $(N_1,q^{\ell-1})$-arc$,$ 
with $N_1=N+q^\ell-q^{\ell-1};$  
   \end{enumerate}
Set $H:=\{B\in\Fql :x^{q^{2r}}+x=B^{q^r}\, 
{\rm has\, a\, solution\, in}\, \Fql\}.$
  \begin{enumerate}
   \item[\rm(5)] Let $p\geq 2$ and $\ell$ be 
even 
with $\ell\equiv 
0\pmod{4}$. Let $H_1$ be a subset of the 
complement of $H$ in $\Fql$ whose 
size is $q^{\ell-1}+q^{r-1}-1.$ Then the 
pointset 
  $$
\cK_2=\cK\cup\{(1:B:0): B\in H_1\}
  $$ 
is a complete 
$(q^{2\ell-1}+q^{\ell-1}+q^{r-1},d)$-arc$;$
  \item[\rm(6)] Let $p=2,$ and $\ell$ be even 
with $\ell\equiv 2\pmod{4}.$ Let 
$H_2$ be a subset of the complement of $H$ in 
$\Fql$ whose 
size is $q^{\ell-1}+q^{r-2}(q-1)-1.$ Then the 
pointset 
  $$
\cK_3=\cK\cup\{(1:B:0): B\in H_2\}
  $$ 
is a complete $(q^{2\ell-1}+q^{\ell-1}+
q^{r-2}(q-1),q^{\ell-1}+q^{r-2}(q-1))$-arc$.$
   \end{enumerate}
   \end{theorem}
   \begin{remark}\label{italianos} Let $q$ be a 
power of an odd prime and $\ell$ be a positive 
even integer. Let 
$B$ be a subset of $\mathbb{F}_{q^{\ell/2}}^*$  of size $b$ with $1\leq b\leq 
b^{\ell/2-1}$. In 
{\cite{italianos}} the following union of Hermitian 
curves over $\Fql$ 
  $$
\cX_B:\,\prod_{\lambda\in B}(\lambda X^{q^{\ell/2}+1}
+XY^{q^{\ell/2}}+
X^{q^{\ell/2}}Y+Z^{q^{\ell/2}+1})=0
  $$ 
is considered. The pointset $\cX_B(\Fql)$ is a 
complete 
$(q^\ell q^{\ell/2}b 
+1,b(q^{\ell/2}+1))$-arc; in particular, 
if $b=q^{\ell/2-1}$ we 
obtain a complete 
$(q^{2\ell-1}+1,q^{\ell-1}+q^{\ell/2-1})$-arc in 
$PG(2,q^\ell)$. For $\ell\geq6$ and 
$\ell\equiv2\pmod{4}$, this arc has the same 
parameters as the arc 
$\cK=\cH(\Fql)$ in Theorem \ref{main-result}(2). 
However, these arcs are not isomorphic. 
In fact, if they were so there would exist a 
collineation $T$ on $PG(2,q^{\ell})$ 
such that $T(\cK)=\cX_B(\Fql)$. By B{\'e}zout's 
Theorem there are at most 
$(q^{\ell-1}+q^{r-1})(q^{\ell-1}+q^{r-3})$ 
points in the intersection of $\cH$ and $\cX_B$, 
which is a contradiction as $\#\cK=q^{2\ell-1}+1$.
   \end{remark}
   \begin{remark}\label{subset} The 
construction of the arcs in Theorem 
\ref{main-result}(5)(6) seem to be not canonical 
in the sense that it 
might depend of the selection of certain subsets 
of $\Fql$. As a matter of fact, we even do not 
know if the smallest case $q=2$ 
and $\ell=4$ would provided with at least two 
non-isomorphic complete 
$(140,12)$-arcs in $PG(2,16)$.
   \end{remark}
	
\subsection*{Acknowledgments.} 
We thank M. Giulietti, 
J.W.P Hirschfeld, G. Korchm\'aros, J. Moyano-Fernandez, and D.
Panario for useful comments. \\ H. Borges was partially supported by FAPESP-Brazil Grant 2011/19446-3, 
B. Motta was partially supported by CAPES-Brazil and CNPq-Brazil, and 
F. Torres was partially supported by CNPq-Brazil Grant 306324/2011-3.


\begin{thebibliography}{HD}
	
	\normalsize
\baselineskip=17pt

\bibitem{hb} H. Borges Filho, \emph{On complete 
$(N,d)$-arcs derived from 
plane curves}, Finite Fields Appl. {15}(1) 
(2009), 82--96.

\bibitem{hb2} H. Borges Filho and R. 
Concei\c{c}\~ao, \emph{Some elementary abelian 
p-extensions 
and a generalization of the Hermitian curve}, 
preprint, (2013).

\bibitem{bul} S.V. Bulygin, \emph{Generalized 
Hermitian codes over GF(2,r)}, IEEE Trans. 
Inform. Theory {52} (2006), 4664--4669.

\bibitem{CO} E. \c{C}ak\c{c}ak and F. 
\"Ozbudak, \emph{Curves related to Coulter's 
maximal curves}, Finite Fields Appl. {14} 
(2008), 209--220.

\bibitem{c1} R.S. Coulter, \emph{Explicit 
evaluations of some Weil sums}, Acta 
Arith. {83} (1998), 241--251.

\bibitem{c2} R.S. Coulter, \emph{Further 
evaluations of Weil sums}, Acta Arith. 
{86} (1998), 217--226.

\bibitem{c4} R.S. Coulter, \emph{On the 
evaluation of a class of Weil sums in 
characteristic 2}, New Zealand J. Math. {28} (1999), 171--184.

\bibitem{c3} R.S. Coulter, \emph{The number of 
rational points of a class of Artin--Schreier 
curves}, Finite Fields Appl. {\bf 8} (2002), 
397--413.

\bibitem{garcia-ozbudak} A. Garcia and F. 
\"Ozbudak, \emph{Some maximal function fields 
and additive polynomials}, Comm. Algebra, {35} (2007), 1553--1566.

\bibitem{gshermitiana} A. Garcia and H. 
Stichtenoth, \emph{A class of polynomials over 
finite fields}, Finite Fields Appl. {5} 
(1999), 424--435.

\bibitem{GK} M. Giulietti and G. Korchm\'aros, 
\emph{A new family of maximal curves over a 
finite field}, Math. Ann. {343} (2009), 
229--245.

\bibitem{italianos} M. Giulietti, F. Pambianco, 
F. Torres and E. Ughi, \emph{On complete arcs 
arising from plane curves}, Des. Codes 
Cryptogr. {25}(3) (2002), 237--246.

\bibitem{hefez-voloch} A. Hefez and J.F. 
Voloch, \emph{Frobenius nonclassical curves}, 
Arch. Math. {54} (1990), 263--273.

\bibitem{hirschfeld} J.W.P. Hirschfeld, 
``Projective Geometries over Finite Fields", 
second edition, Oxford University Press, Oxford 
(1998).

\bibitem{HKT} J.W.P. Hirschfeld, G. 
Korchm\'aros and F. Torres, ``Algebraic Curves 
over a Finite Field", Princeton University 
Press, Princeton and Oxford (2008).

\bibitem{hirschfeld-voloch} J.W.P. Hirschfeld 
and J.F. Voloch, \emph{The characterization of 
elliptic curves over finite fields}, J. 
Austral. Math. Soc. Ser. A {45} (1988), 
275--286.

\bibitem{castle} C. Munuera, A. Sep\'{u}lveda 
and F. Torres, \emph{Algebraic Geometry codes 
from Castle curves}, Lecture Notes in Comput. 
Sci. {5228} (2008), Springer-Verlag Berlin 
Heidelberg, 117--127.

\bibitem{mst} C. Munuera, A. Sep\'{u}lveda and 
F. Torres, \emph{Generalized Hermitian codes}, 
Des. Codes Cryptogr. {69}(1) (2013), 123--130.

\bibitem{nl} H. Niederreiter and R. Lidl, 
``Finite Fields", Encyclopedia of Mathematics 
and its Applications, Addison-Wesley Publishing 
Company, Inc. (1983).

\bibitem{sti} H. Stichtenoth, ``Algebraic 
Function Fields and Codes", Springer-Verlag, 
Berlin (1993).

\bibitem{saeed} S. Tafazolian, \emph{A note on certain maximal hyperelliptic curve}, Finite 
Fields Appl. {18} (2012), 1013--1016.

\bibitem{TV} M.A. Tsfasman and S.G. Vl\v{a}du\c{t}, 
``Algebraic-Geometric Codes", Kluwer, Amsterdam 
(1991).

\bibitem{wolfmann} J. Wolfmann, \emph{The number 
of points on certain algebraic curves over 
finite fields}, Comm. Algebra, {17}(8) 
(1989), 2055--2060.

   \end{thebibliography}
   \end{document}